\documentclass{elsarticle}

\pdfoutput=1

\usepackage{graphicx} 
\usepackage{amsfonts, amsthm, amsmath}
\usepackage{hyperref}
\usepackage{geometry}
\usepackage{wrapfig}

\usepackage{calc} 
\usepackage[shortlabels]{enumitem} 
\newlength{\circlabelwidth}
    \setlist{nosep,before={\parskip=0pt plus 2pt},after={\parskip=0.5em plus 2pt}}
    \setlist[itemize]{labelindent=10pt,labelwidth=\circlabelwidth,leftmargin=!,label=$\circ$}

\usepackage{tikz} 
\usetikzlibrary{3d}
\usetikzlibrary{calc} 
\usetikzlibrary{math}

\newcommand{\R}{\mathbb{R}}
\newcommand{\conv}{\operatorname{conv}}

\newtheorem{theorem}{Theorem}
\newtheorem{lemma}[theorem]{Lemma}
\newtheorem{proposition}[theorem]{Proposition}

\theoremstyle{remark}
\newtheorem*{remark}{Remark}

\theoremstyle{definition}

\mathcode`l="8000
\begingroup
\makeatletter
\lccode`\~=`\l
\DeclareMathSymbol{\lsb@l}{\mathalpha}{letters}{`l}
\lowercase{\gdef~{\ifnum\the\mathgroup=\m@ne \ell \else \lsb@l \fi}}%
\endgroup

\date{January 2025}

\begin{document}

\title{Piercing intersecting convex sets}

\author[1,2]{Imre B\'ar\'any}
\ead{barany.imre@renyi.hu}

\author[3]{Travis Dillon}
\ead{travis.dillon@mit.edu}

\author[1,4]{Dömötör Pálvölgyi}
\ead{domotor.palvolgyi@ttk.elte.hu}

\author[1]{Dániel Varga}
\ead{daniel@renyi.hu}

\affiliation[1]{
    organization={HUN-REN Alfréd Rényi Institute of Mathematics}, 
    addressline={13 Reáltanoda Street}, 
    city={Budapest}, 
    postcode={1053}, 
    country={Hungary}
}
\affiliation[2]{
    organization={Department of Mathematics, University College London}, 
    addressline={Gower Street}, 
    city={London}, 
    postcode={WC1E 6BT}, 
    country={UK}
}
\affiliation[3]{
    organization={Massachusetts Institute of Technology}, 
    addressline={77 Massachusetts Avenue}, 
    city={Cambridge}, 
    state={MA}, 
    country={USA}
}
\affiliation[4]{
    organization={ELTE Eötvös Loránd University}, 
    city={Budapest}, 
    country={Hungary}
}


\begin{abstract}
Assume two finite families $\mathcal A$ and $\mathcal B$ of convex sets in $\mathbb{R}^3$ have the property that $A\cap B\ne \emptyset$ for every $A \in \mathcal A$ and $B\in \mathcal B$. Is there a constant $\gamma >0$ (independent of $\mathcal A$ and $\mathcal B$) such that there is a line intersecting $\gamma|\mathcal A|$ sets in $\mathcal A$ or $\gamma|\mathcal B|$ sets in $\mathcal B$? This is an intriguing Helly-type question from a paper by Mart\'{i}nez, Roldan and Rubin. We confirm this in the special case when all sets in $\mathcal A$ lie in parallel planes and all sets in $\mathcal B$ lie in parallel planes; in fact, all sets from one of the two families has a line transversal.
\end{abstract}

\maketitle

\section{Introduction and main results}\label{sec:intro}

In a paper on extensions of the Colorful Helly Theorem, Mart\'{i}nez, Roldan and Rubin~\cite{MRR} ask the following: Suppose two finite families $\mathcal A$ and $\mathcal B$ of convex sets in $\R^3$ have the property that $A\cap B\ne \emptyset$ whenever $A \in \mathcal A$ and $B\in \mathcal B$. Is it true then that there is a line intersecting a fixed positive fraction of the sets in $\mathcal A$ or in $\mathcal B$? The best we know is that it is true in some special cases. For instance, B\'ar\'any \cite{bar} confirmed that when all sets in $\mathcal A$ and $\mathcal B$ are cylinders, or if all sets have a bounded aspect ratio, then it is true. In general, however, this question is wide open.

A different special case of this question, namely the case of \emph{vertical polygons}, was raised independently by Andreas Holmsen and G\'eza T\'oth (personal communication). A vertical polygon in $\R^3$ is a convex polygon that lies in a plane orthogonal to the $xy$-plane. Is there a real number $\gamma >0$ such that, whenever both $\mathcal A$ and $\mathcal B$ consist of vertical polygons, there is a line intersecting a $\gamma$-fraction of the sets in $\mathcal A$ or a $\gamma$-fraction of those in $\mathcal B$? One motivation for studying this question is that the aspect ratio of a vertical polygon is infinite, so it is as far as possible from the case of convex bodies with bounded aspect ratio; perhaps a solution in this case could shed light on the general problem. Even this special case, however, remains open.

Our main result, Theorem~\ref{th:Helly} below, states there is a line intersecting \emph{all} sets of $\mathcal A$ or of $\mathcal B$, provided the vertical polygons in $\mathcal A$ lie in parallel planes and the vertical polygons in $\mathcal B$ also lie in parallel planes. We also prove, under the same condition, that we can restrict the location of the piercing line: Either there is a line in the plane of some $A\in \mathcal A$ intersecting $\frac{1}{6} |\mathcal B|$ sets in $\mathcal B$, or there is a line in the plane of some $B\in \mathcal B$ intersecting $\frac{1}{6} |\mathcal A|$ sets in $\mathcal A$. This is Theorem~\ref{th:fracH}, which is stated and proved in Section~\ref{sec:frac}. 

We conclude the paper in Section~\ref{sec:dim} with a partial extension of our results to higher dimensions.

\section{Theorem~\ref{th:Helly} and its proof}\label{sec:main-thm}

We start with two families, $\mathcal A$ and $\mathcal B$, consisting of vertical polygons such that all the polygons in $\mathcal A$ are in parallel planes, all the polygons in $\mathcal B$ are in parallel planes, and $A \cap B \neq \emptyset$ whenever $A \in \mathcal A$ and $B \in \mathcal B$. For each pair $(A,B)$, take an arbitrary point $P(A,B)$ in $A \cap B$. If we replace $A$ by $\conv\{P(A,B): B \in \mathcal B\}$ and replace $B$ by $\conv\{P(A,B): A \in \mathcal A\}$, this new collection of vertical polygons is still parallel and pairwise intersecting. If one of these new families has a line transversal, then the corresponding original family does, as well. Moreover, the question is invariant under non-degenerate affine transformations, so we can assume that the planes containing sets in $\mathcal A$ are parallel with the $yz$-plane, and the planes containing sets in $\mathcal B$ are parallel with the $xz$-plane.

Based on these reductions, we formulate our problem in a more convenient notation.
Let $x \in \R^n, y \in \R^m$ and $Z \in \R^{n \times m}$, and for each $i\in [n],j \in [m]$, let $P_{ij} := (x_i, y_j, z_{ij}) \in \R^3$. We form convex sets $A_i := \conv(\{P_{ij} : j \in [m] \} )$ and $B_j := \conv(\{P_{ij} : i \in [n] \} )$. By construction, every $A_i$ is contained in a plane parallel to the $yz$-plane, every $B_j$ is contained in a plane parallel to the $xz$-plane, and every $A_i$ intersects every $B_j$. (In this setup, the point $P_{ij}$ corresponds to the point $P(A_i, B_j)$ in the previous paragraph.)  This setup is illustrated in Figure \ref{fig:3x3}.

\begin{figure}[b]
    \centering
    \begin{tikzpicture}[x={(-10:1cm)}, y={(180+30:0.8cm)}, z={(0,1cm)}]
        \foreach \x in {1,2,3}
            \draw[gray] (\x,1) -- (\x,3) (1,\x) -- (3,\x);
        \foreach \x/\y/\z in {1/1/0.4, 2/1/1.5, 3/1/1, 1/2/0.5, 2/2/0.5, 3/2/0.1, 1/3/0, 2/3/0.7, 3/3/0}{
            \coordinate (\x\y) at (\x,\y,\z) {};
            \draw[dotted] (\x,\y,0) -- (\x,\y,\z);
        }
        \foreach \a in {1,2,3}{
            \fill[opacity=0.2,red] (\a1) -- (\a2) -- (\a3) -- cycle;
            \fill[opacity=0.2,blue] (1\a) -- (2\a) -- (3\a) -- cycle;
        }
        \draw[->] (0,0,0) -- (0,0,1.75) node[above]{$z$};
        \draw[->] (0,0,0) -- (0,3.5,0) node[below left]{$y$};
        \draw[->] (0,0,0) -- (3.5,0,0) node[right]{$x$};
    \end{tikzpicture}
    \caption{An illustration of the $n=m=3$ case. The red sets form the family $\mathcal A$ and the blue sets form the family $\mathcal B$.}
    \label{fig:3x3}
\end{figure}
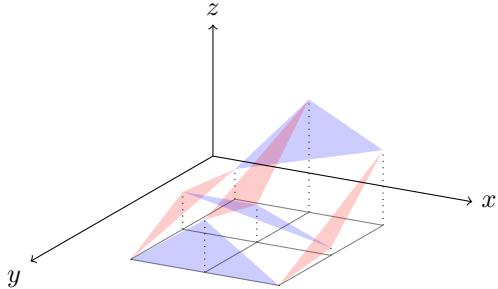

\begin{theorem}\label{th:Helly}
With the setup of the previous paragraph, either
\begin{itemize}
    \item there is a real number $x_0 \in \R$ and a line $l_x$ in the plane $(x_0, \,\cdot\,, \,\cdot\,)$ intersecting all sets $B_j$, or
    \item there is a real number $y_0 \in \R$ and a line $l_y$ in the plane $(\,\cdot\,, y_0, \,\cdot\,)$ intersecting all sets $A_i$.
\end{itemize} 
\end{theorem}

Because the problem is affine-invariant, this proves the theorem for any two collections of vertical polygons that live in parallel planes. 

\begin{proof}
We start by writing up the intended conclusions as a system of linear inequalities which has a solution if and only if a piercing line exists.

Let us do this for the first possibility. Let the line $l_x$ be parametrised as $\{(x_0, y, ay+z_0): y \in \R \}$. The claim that $l_x$ pierces $B_j$ is equivalent to the existence of barycentric coordinates $(\beta_{ij} : i \in [n])$ such that $\sum_i \beta_{ij} P_{ij} = (x_0, y_j, ay_j+z_0)$. 

Our system of linear inequalities specifying a piercing line $(x_0, y, ay+z_0)$ is then

\begin{align}
    \sum_i \beta_{ij} x_i &= x_0 & \forall j \in [m], \tag{x-piercing}\label{primal:x}\\
    \sum_i \beta_{ij} z_{ij} &= ay_j + z_0 & \forall j \in [m], \tag{z-piercing}\label{primal:z}\\
    \sum_i \beta_{ij} &= 1 & \forall j \in [m], \tag{barycentric}\label{primal:barycentric}\\
    \beta_{ij} &\geq 0 & \forall i\in [n],j \in [m]. \tag{nonnegative}\label{primal:nonnegative}
\end{align}

We now apply Farkas' Lemma to write up a system that is unsolvable in dual variables $U \in \mathbb{R}^{3 \times m}$ if and only if the above system is solvable in $(\beta, a, x_0, z_0)$.

\begin{align}
u_{1j} x_i + u_{2j} z_{ij} + u_{3j} &\geq 0 & \forall i\in [n],j \in [m], \tag{$\beta$}\label{dual:beta}\\
\sum_j u_{2j} y_j &= 0, \tag{$a$}\label{dual:a}\\
\sum_j u_{1j} &= 0, \tag{$x_0$}\label{dual:x}\\
\sum_j u_{2j} &= 0, \tag{$z_0$}\label{dual:z}\\
\sum_j u_{3j} &< 0. \tag{infeasible}\label{dual:infeasible}
\end{align}

Below we describe the correspondence between the parts of the two systems of inequalities. The real vectors $u_{1j}, u_{2j}, u_{3j}$ correspond to equalities (\ref{primal:x}), (\ref{primal:z}), (\ref{primal:barycentric}), respectively. The $n \times m$ inequalities of (\ref{dual:beta}) correspond to the $n \times m$ non-negative variables of $\beta$. The equalities (\ref{dual:a}), (\ref{dual:x}), and (\ref{dual:z}) correspond to the real variables $a, x_0, z_0$, respectively.

Our next step is to combine the above dual system with the analogous system written up for the $l_y$ line. The combined system has $x, y, Z, U, V$ as variables, and is bilinear, non-semidefinite. This is in contrast with its two constituents that are linear when $x, y, Z$ are treated as parameters.

The complete system we get is:\vspace{\baselineskip}

\noindent
$\exists x, y \in \mathbb{R}^{n},\ \exists Z \in \mathbb{R}^{n \times m},\ \exists U\in \mathbb{R}^{3 \times m}, V \in \mathbb{R}^{3 \times n}:$
\begin{align}
u_{1j} x_i + u_{2j} z_{ij} + u_{3j} &\geq 0 & \forall i\in [n],j \in [m], \tag{x:$\beta$}\label{combinedx:beta}\\
v_{1i} y_j + v_{2i} z_{ij} + v_{3i} &\geq 0 & \forall i\in [n],j \in [m], \tag{y:$\beta$}\label{combinedy:beta}\\
\sum_j u_{2j} y_j &= 0, \tag{x:$a$}\label{combinedx:a}\\
\sum_i v_{2i} x_i &= 0, \tag{y:$a$}\label{combinedy:a}\\
\sum_j u_{1j} = \sum_i v_{1i} = \sum_j u_{2j} = \sum_i v_{2i} &= 0, \tag{x:$x_0$, y:$x_0$, x:$z_0$, y:$z_0$}\label{combinedxy:xz}\\
\sum_j u_{3j} &< 0, \tag{x:infeasible}\label{combinedx:infeasible}\\
\sum_i v_{3i} &< 0. \tag{y:infeasible}\label{combinedy:infeasible}
\end{align}

Having the combined system of inequalities at hand, we now finish the proof. We demonstrate that this system is unsolvable rather directly, by writing up weighted sums of our inequalities until a contradiction is reached. This implies that one of the original systems must be solvable; in other words, either $\mathcal A$ or $\mathcal B$ has a piercing line.

\medskip
    \textbf{Remark.} The argument could be formulated in the framework of Farkas' lemma: if we treat $U, V$ as parameters and $x, y, Z$ as variables, we can give a closed-form solution to the dual of the resulting linear system of inequalities. In an interesting contrast with this, if we treat $x, y, Z$ as parameters and $U, V$ as variables, the dual formulation is equivalent to presenting a piercing line, and we are not aware of any simple formula for a dual solution.
    \medskip

Let us introduce the three variables
\[
    x'_i:=v_{2i}x_i, \qquad
    y'_j:=u_{2j}y_j, \qquad
    z'_{ij}:=u_{2j}v_{2i} z_{ij}
\]
and the four sets
\[\begin{array}{c @{\hspace{2em}} c}
    I^+=\{i : v_{2i}\ge 0\} & J^+=\{j : u_{2j}\ge 0\} \\
    I^-=\{i : v_{2i}< 0\} & J^-=\{j : u_{2j}< 0\}\rlap{.}
\end{array}\]
By multiplying (\ref{combinedx:beta}) by $v_{2i}$ we obtain
\begin{align}
    \forall i\in I^+, \ \forall j&: u_{1j} x_i' + z_{ij}' + v_{2i}u_{3j} \geq 0,\label{eqxp}\\
    \forall i\in I^-, \ \forall j&: u_{1j} x_i' + z_{ij}' + v_{2i}u_{3j} \le 0.\label{eqxn}
\end{align}
Similarly, by multiplying (\ref{combinedy:beta}) by $u_{2j}$ we obtain
\begin{align}
    \forall j\in J^+, \ \forall i &: v_{1i} y_j' + z_{ij}' + u_{2j}v_{3i} \geq 0,\label{eqyp}\\
    \forall j\in J^-, \ \forall i&: v_{1i} y_j' + z_{ij}' + u_{2j}v_{3i} \leq 0.\label{eqyn}
\end{align}

By comparing (\ref{eqxp}) and (\ref{eqyn}), we obtain
\begin{align}
    \forall i\in I^+, \ \forall j\in J^-&: u_{1j} x_i' - v_{1i} y_j' + v_{2i}u_{3j} -u_{2j}v_{3i} \geq 0.\label{eqc1}
\end{align}
Similarly, from (\ref{eqxn}) and (\ref{eqyp}), we obtain
\begin{align}
    \forall i\in I^-, \ \forall j\in J^+&: -u_{1j} x_i' + v_{1i} y_j' - v_{2i}u_{3j} +u_{2j}v_{3i} \geq 0.\label{eqc2}
\end{align}

We now sum the last two inequalities over each of their ranges and add them together; this will give us a contradiction. To find it, we look separately at each of the four parts of this sum, one for each summand in the inequalities \eqref{eqc1} and \eqref{eqc2}.

The first part is
\[
    \sum_{i\in I^+}\sum_{j\in J^-} u_{1j} x_i' - \sum_{i\in I^-}\sum_{j\in J^+} u_{1j} x_i',
\]
which we deal with by introducing the shorthand
\[\begin{array}{c @{\hspace{2em}} c}
    u_1^+=\sum_{j\in J^+} u_{1j} & x^+=\sum_{j\in J^+} x_{j}' \\
    u_1^-=\sum_{j\in J^-} u_{1j} & x^-=\sum_{j\in J^-} x_{j}'\rlap{.}
\end{array}\]
From (\ref{dual:x}), we have $u_1^+=-u_1^-$, and from (\ref{combinedy:a}), we have $x^+=-x^-$.
Thus, the first part of the sum is
$u_1^-x_1^+-u_1^+x_1^-=u_1^-x_1^+-u_1^-x_1^+=0$.

The second part of the sum also vanishes; the proof is the same, using the variables
\[\begin{array}{c @{\hspace{2em}} c}
    v_1^+=\sum_{i\in I^+} v_{1i} & y^+=\sum_{j\in J^+} y_{j}' \\
    v_1^-=\sum_{i\in I^-} v_{1i} & y^-=\sum_{j\in J^-} y_{j}'\rlap{.}
\end{array}\]

To deal with the third part of the sum, we introduce
\[\begin{array}{c @{\hspace{2em}} c}
    v_2^+=\sum_{i\in I^+} v_{2i} & u_3^+=\sum_{j\in J^+} u_{3j} \\
    v_2^-=\sum_{i\in I^-} v_{2i} & u_3^-=\sum_{j\in J^-} u_{3j}\rlap{;}
\end{array}\]
thus $v_2^++v_2^-=0$ and $u_3^++u_3^-=0$. The third part of the sum is
$v_2^+u_3^--v_2^-u_3^+=v_2^+u_3^-+v_2^+u_3^+$, which is negative by (\ref{combinedx:infeasible}), unless $v_2^+=0$. If $v_2^+ = 0$, then $v_{2i}=0$ for all $i$, in which case the sum of (\ref{combinedy:beta}) over all $i$ is $\sum_i v_{3i}=0$, which contradicts (\ref{combinedy:infeasible}).

In a similar way, we can obtain that the fourth part of the sum is also negative. But this contradicts the fact that the sum of the four parts should be nonnegative; therefore the system is unsolvable.
\end{proof}

\subsection{Comments on the proof}

In this section, we present some remarks on the peculiarity of the above proof: It is a non-constructive existence proof, which proceeds by using linear programming duality twice.

To highlight the unusual structure of this proof, we now abstract away the concrete details. In what follows, $a$ corresponds to a configuration of convex sets $(\mathcal A, \mathcal B)$, while $b$ roughly corresponds to the piercing line whose existence is stated in Theorem~\ref{th:Helly}. We will clarify the exact semantics of $b$ after the proof outline.

Rather than constructing a $b$ for every $a$, its existence is non-constructively proven in the following way: Let $c$ be a potential witness to the fact that $b$ does not exist for a given $a$. For a given $c$, we construct a witness $d$ to the fact that $c$ is not in fact a witness for any given $a$. To elaborate:

\begin{itemize}
    \item We would like to prove that $\forall a: \exists b: S_1(a, b)$. ($a$ and $b$ are real vectors; $S_1$ is a bilinear system of inequalities.)
    \item We treat $a$ as a parameter and $b$ as a variable, and apply Farkas' lemma to get equivalent statement $\forall a: \neg \exists c: S_2(a, c)$.
    \item We switch quantifiers: $\forall c: \forall a: \neg S_2(a, c)$.
    \item We now treat $c$ as a parameter and $a$ as a variable, and apply Farkas' lemma again: $\forall c: \exists d: S_3(c, d)$.
    \item The above steps were all equivalences. Hence $\forall a: \exists b: S_1(a, b)$ if and only if $\forall c: \exists d: S_3(c, d)$.
    \item We prove $\forall c: \exists d: S_3(c, d)$ by explicitly constructing a witness $d$ for any $c$.
\end{itemize}\vspace{\baselineskip}

\noindent
The correspondence between the above scheme and the actual notation in the proof of Theorem~\ref{th:Helly} is as follows:
\begin{itemize}
    \item $a$ corresponds to our configuration of sets $(\mathcal A, \mathcal B)$, parametrized by $(x, y, Z)$.
    \item $b$ corresponds to the pair of lines $(l_x, l_y)$. Its exact parametrization appears in the proof only implicitly, but it can be described by $(x_0, a, z_0, \beta)$ in the $x$ direction, $(y_0, a', z'_0, \beta')$ in the $y$ direction, with an extra slack variable $s$ whose sign determines which of the two directions is supposed to be piercing.
    \item $S_1(a, b)$ is the bilinear system of inequalities that states that $b$ pierces $a$.
    \item $c$ is the potential witness of $b$'s non-existence, parametrized by $(U, V)$.
    \item $S_2(a, c)$ is the dual system stating that $c$ is a dual witness for $a$, disproving the existence of a piercing line $b$ for the given $a$.
    \item $d$ is a weighting on the inequalities of the dual system.
    \item $S_3(c, d)$ states that the $d$-weighted sum of the inequalities of the dual system $S_2(a, c)$ leads to a contradiction, which proves that $c$ cannot be a dual witness for any $a$.
\end{itemize}

We have not succeeded in streamlining this seemingly roundabout proof structure, and we now believe that such a streamlining is, in fact, not possible. That is, we formulate the following informal meta-mathematical conjecture: Any proof of Theorem~\ref{th:Helly} is necessarily non-constructive.

Let us note, however, that we \emph{can} construct a piercing line for any given configuration $a$ by solving a linear program. Using Megiddo's algorithm for constant dimensional linear programming~\cite{Meg}, we can construct the piercing line in linear time. In that sense, $b$ can be constructed. Another sense in which $b$ can be constructed is that for a given fixed $n$, the space of possible $a$ configurations can be partitioned into finitely many simplices such that a piecewise linear map assigns a piercing $b$ to any $a$. That is an explicit (gigantic) formula for $b$. Our informal meta-mathematical conjecture states that there is no such explicit formula independent of $n$. Of course, to fully formalize this conjecture, we would have to specify the set of allowed operations on vectors (such as sum, maximum, argmax, argsort, etc.).

\section{Fractional line transversals}\label{sec:frac}

In this section, we maintain the same notation as established at the beginning of Section \ref{sec:main-thm}. For convenience, we also assume that $x_1<\ldots< x_n$ and $y_1<\ldots<y_m$. As before, $Z\in \R^{n \times m}$ and $P_{i,j}=(x_i,y_j,z_{i,j})$; and $A_i=\conv\{P_{i,j}: j\in [m]\}$ and $B_j=\conv\{P_{i,j}: i\in [n]\}$. 

\begin{theorem}\label{th:fracH} With the above notation, either there is a line in the plane $(x_i,\,\cdot\,,\,\cdot\,)$ for some $i\in [n]$ intersecting $\frac m6$ sets out of $B_1,\ldots, B_m$, or there is a line in the plane $(\,\cdot\,,y_j,\,\cdot\,)$ for some $j\in [m]$ intersecting $\frac n6$ sets out of $A_1,\ldots ,A_n$.
\end{theorem}

\begin{remark}
    The constant $\frac16$ depends on the bound in the fractional Helly theorem.
    By applying Kalai's better bound for the fractional Helly theorem \cite{Kal}, we may replace $1/6$ by the slightly larger number $1-\sqrt[3]{1/2} \approx 0.206$.
\end{remark}

Unlike Theorem 1, Theorem \ref{th:fracH} only gives a fractional transversal; however, it guarantees that this transversal lies in one of the planes $(x_i,\,\cdot\,,\,\cdot\,)$ or $(\,\cdot\,,y_j,\,\cdot\,)$.

The rest of this section comprises a proof of Theorem \ref{th:fracH}. We will start with the case $n=m=3$ and extend this to Theorem \ref{th:fracH} using the Fractional Helly Theorem. To simplify notation, we let $H_x^i$ denote the plane $(x_i,\,\cdot\,,\,\cdot\,)$ and $H_y^j$ denote the plane $(\,\cdot\,, y_j, \,\cdot\,)$. (So $A_i \subset H_x^i$ and $B_j \subset H_y^j$.)

\begin{lemma}\label{thm:3x3_grid}
    If $n=3$, there is either a line $\ell_x$ in $H_x^2$ intersecting all sets $B_j$ or there is a line $\ell_y$ in $H_y^2$ intersecting all sets $A_i$.
\end{lemma}

We note that Holmsen's paper~\cite{Holm} extending the Montejano--Karasev theorem~\cite{KaMo} already implies that one of the two families in Lemma \ref{thm:3x3_grid} can be pierced by a line; however, we also prove something about the \emph{location} of the line, which will be necessary to prove Theorem \ref{th:fracH}. 

 
\begin{proof}
    Let $L_{i,j}$ be the vertical line defined by $(x_i,y_j,\,\cdot\,)$. Let $P_x$ be the intersection of the segment $P_{2,1}P_{2,3}$ with $L_{2,2}$ and let $P_y$ be the intersection of the segment $P_{1,2}P_{3,2}$ with $L_{2,2}$. (So $P_x$ is contained in $A_2$ and $P_y$ is contained in $B_2$.)

    Consider the spatial quadrangle $Q$ with vertices at $P_{i,j}$ for $i,j \in \{1,3\}$. We form $\ell_A$ as the line segment that passes through both points in $Q \cap H_x^2$; we also form $\ell_B$ as the line segment that passes through both points in $Q \cap H_y^2$. (See Figure \ref{fig:quadrangle}.) Both $\ell_A$ and $\ell_B$ intersect $L_{2,2}$; in fact, they intersect at the \emph{same} point, given by a weighted average of the four points $P_{i,j}$ with $i,j \in \{1,3\}$. (This is the place in the argument where we need the three planes in each family to be parallel---these lines would not necessarily intersect otherwise.) Denote this common intersection by $R$.
    

    By convexity, the segment $P_{2,2} P_x$ is contained in $A_2$, while the interval $P_{2,2} P_y$ is contained in $B_2$. Consider the set $\mathcal S_y$ of segments in $H_x^2$ that intersect $B_1$ and $B_3$. The collection $\{ S \cap L_{2,2} : S \in \mathcal S_y\}$ is convex and contains both the points $R$ and $P_x$, so it contains the segment $R P_x$. Similarly, let $\mathcal S_x$ be the set of segments in $H_y^2$ that intersect $A_1$ and $A_3$. The set $\{ S \cap L_{2,2} : S \in \mathcal S_x\}$ is convex and contains the points $P_y$ and $R$, so it contains the segment $R P_y$.
    
    If the segment $R P_x$ intersects the segment $P_{2,2} P_y$, then there is a line contained in $H_x^2$ that pierces $B_1,B_2,B_3$. Similarly, if the segment $R P_y$ intersects the segment $P_{2,2} P_x$, then there is a line contained in $H_y^2$ that pierces $A_1, A_2, A_3$.
    But one of these must occur, as in any topological embedding of $\mathbb S^1$ to $\R$ some two antipodal points are mapped to the same point; so one could say that we are using the $d=1$ case of the Borsuk--Ulam theorem when we map onto $RP_xP_{2,2}P_y$, but of course this one dimensional case can be proved more simply by checking a few cases.
    (See also the proof of Lemma \ref{thm:subset-points} and the remark after it.)
\end{proof}

\begin{figure}[t]
    \centering
    \begin{tikzpicture}[x={(-10:1cm)}, y={(180+30:0.8cm)}, z={(0,1cm)}]
        \foreach \x in {1,2,3}
            \draw[gray] (\x,1) -- (\x,3) (1,\x) -- (3,\x);
        \foreach \x/\y/\z in {1/1/1.5, 3/1/2, 1/3/1.5, 3/3/1.2}{
            \coordinate (\x\y) at (\x,\y,\z) {};
            \draw[dotted] (\x,\y,0) -- (\x,\y,\z);
        }
        \draw (11) -- (13) -- (33) -- (31) -- cycle;
        \draw[blue, thick] ($(11)!0.5!(13)$) coordinate (r1) -- ($(31)!0.5!(33)$) coordinate (r2);
        \draw[red, thick] ($(11)!0.5!(31)$) -- ($(13)!0.5!(33)$);
        \draw[dotted] (2,2,0) -- ($(r1)!0.5!(r2)$) node[above]{\footnotesize $R$\ \null};
    \end{tikzpicture}
    \caption{The quadrangle $Q$, the lines \textcolor{red}{$\ell_A$} and \textcolor{blue}{$\ell_B$}, and their intersection at $R$.}
    \label{fig:quadrangle}
\end{figure}
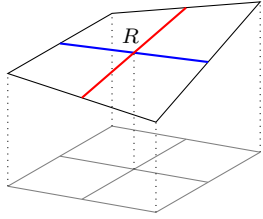

\begin{wrapfigure}{r}{0.25\textwidth}
    \centering
    \begin{tikzpicture}[y={(40:0.3cm)}, x={(-10:0.4cm)}, z={(0,2.2cm)},scale=1.45]
        \tikzmath{\n=0.75;}
        \tikzmath{\offset=0.3;}
        \foreach \x/\y/\z/\w in {1/1/3/1, 1/2/3/2, 1/3/3/5, 1/1/1/3, 2/1/2/4, 3/1/3/5}
            \draw[gray] (\x,\y,0) -- (\z,\w,0); 
        \foreach \x/\y/\z in {1/1/0, 2/1/0, 3/1/0, 1/2/0, 3/2/{\n/4}, 1/3/0, 2/4/{\n/2}, 3/5/\n}%
        {
            \coordinate (\x\y) at (\x,\y,{\offset+\z}) {};
            \draw[dotted,opacity=0.75] (\x,\y,0) -- (\x\y);
        }
        \draw[red, opacity=0.75, thick] (11) -- (13) (31) -- (35);
        \draw[blue, opacity=0.75, thick] (11) -- (31) (13) -- (35);
    \end{tikzpicture}
\end{wrapfigure}

In Lemma \ref{thm:3x3_grid}, it is crucial that both families of planes are parallel. If this is not the case, then the claim may be false, as the following example shows. Define the points
\[
    x_{11} = (1,1,0)\quad
    x_{31} = (3,1,0)\quad
    x_{13} = (1,3,0)\quad
    x_{35} = (3,5,1)\quad
    x_{22} = (2,2,{\textstyle \frac{1}{7}}),
\]
then the points $x_{11}, x_{13}, x_{31}, x_{35}$ form the quadrangle shown at the right. However, for this grid, the line $\ell_B$ lies below $\ell_B$ (as defined in the proof of Lemma \ref{thm:3x3_grid}), while $x_{22}$ lies between the two. Given 
\[
\begin{array}{l @{\qquad} l}
    A_1 = \conv(x_{11}, x_{13}) & B_1 = \conv(x_{11}, x_{31})\\
    A_2 = \conv(\ell_A, x_{22}) & B_2 = \conv(\ell_B, x_{22})\\
    A_3 = \conv(x_{31}, x_{35}) & B_3 = \conv(x_{13}, x_{35}),
\end{array}
\]
then $A_i \cap B_j \neq \emptyset$ for every $i,j \in \{1,2,3\}$ and the sets in $\mathcal A = \{A_1, A_2, A_3\}$ lie in parallel planes, but the sets in $\mathcal B = \{B_1,B_2,B_3\}$ lie in three planes, only two of which are parallel.
But $\ell_A$ is the only line contained in $B_2$'s plane that can pierce $A_1$ and $A_3$, and $\ell_A$ does not pierce $A_2$. Similarly, there is no line in $A_2$'s plane that pierces $\mathcal B$.

We now return to the proof of Theorem \ref{th:fracH}. A theorem of Santal\'o~\cite{San} from 1942 (see also \cite{DGK} and\cite{Bar}) states that: given a finite collection of parallel line segments in the plane such that every triple of these segments has a line transversal, there is a line that intersects all segments. The proof is based on Helly's theorem using the fact that the set of lines intersecting a fixed vertical segment (if parametrized suitably) forms a convex set in the plane. Applying the fractional Helly theorem \cite{KaL} to this situation, instead of Helly's theorem, yields the following lemma.

\begin{lemma}[Fractional Helly for vertical line segments]\label{thm:fractional-helly-segments}
    Suppose that $\mathcal L$ is a collection of parallel line segments in the plane such that at least $\alpha \binom{|\mathcal L|}{3}$ triples of these segments can be stabbed by a line. Then there is a set of $\frac{\alpha}3|\mathcal L|$ segments in $\mathcal L$ that can be stabbed by a single line.
\end{lemma}

We now have all the pieces to finish the proof.

\begin{proof}[Proof of Theorem \ref{th:fracH}]
The points $(x_i,y_j)$ form an $n \times m$ grid in the plane $z=0$, which we will call $G$. A triple $(j_1,j_2,j_3)$ with $1\le j_1<j_2<j_3\le m$ is called \emph{$x_i$-good} for some $i \in [n]$ if there is a line in the plane $(x_i,\,\cdot\,,\,\cdot\,)$ intersecting $B_{j_1},B_{j_2},B_{j_3}$. Analogously, the triple $(i_1,i_2,i_3)$ is \emph{$y_j$-good} if there is a line in the plane $(\,\cdot\,,y_j,\,\cdot\,)$ intersecting $A_{i_1},A_{i_2},A_{i_3}$. Lemma \ref{thm:3x3_grid} says that in every $3 \times 3$ subgrid $\{(x_{i_s},y_{j_t}), s,t=1,2,3\}$ of $G$, either $(j_1,j_2,j_3)$ is $x_{i_2}$-good
 or $(i_1,i_2,i_3)$ is $y_{j_2}$-good. 

Assume next that $\delta$ is the smallest number such that the number of $x_i$-good triples is at most $\delta \binom{m}{3}$ for every $i\in [n]$, and the number of $y_j$-good triples is at most $\delta \binom{m}{3}$ for every $j\in [m]$. We show first that

\begin{equation}\label{eq:Delta}
\delta \ge \frac{1}{2}.
\end{equation} 

The proof is by double counting. Any fixed $x_i$-good triple, say $(j_1,j_2,j_3)$, will appear in exactly $(i-1)(n-i)$ $3 \times 3$ subgrids. This gives at most
\[
\delta\binom{m}{3}\sum_{i=2}^{n-1}(i-1)(n-i)=\delta\binom{m}{3}\binom{n}{3}
\]
$3 \times 3$ subgrids that contain an $x_i$-good triple for some $i$. The same arguments with $x$ and $y$ exchanged gives the same upper bound for the number of $3 \times 3$ subgrids that contain a $y_j$-good triple for some $j$.

As the total number of $3 \times 3$ subgrids is $\binom{n}{3}\binom{m}{3}$, Lemma \ref{thm:3x3_grid} implies that there are at least $\binom{n}{3}\binom{m}{3}$ $3\times 3$ subgrids that are $x_i$- or $y_j$-good for some $i$ or $j$. Thus $\binom{n}{3}\binom{m}{3} \leq 2\delta \binom{n}{3}\binom{m}{3}$, which implies the inequality (\ref{eq:Delta}). 

So there are at least $\frac 12 \binom{m}{3}$ $x_i$-good triples for some $i\in [n]$ or there are at least $\frac 12 \binom{n}{3}$ $y_j$-good triples for some $j\in [m]$. The arguments are symmetric, so assume that the latter case occurs. The plane $H=(\,\cdot\,,y_j,\,\cdot\,)$  intersects the sets $A_1,\ldots,A_n$ in parallel segments, and at least half of the triples of these segment have a line transversal. Lemma~\ref{thm:fractional-helly-segments} implies that there is a line in $H$ intersecting $\frac n6$ of the sets $A_1,\ldots,A_n$. 
\end{proof}

\section{Partial extension to higher dimensions}\label{sec:dim}

How do we extend our results to higher dimensions? Informally, we can imagine the setup of vertical sets as taking an $n\times n$ grid in $\R^2$, and choosing a point in $\R^1$ for each intersection point of the grid. To extend to higher dimensions, we take an $n\times n \times \cdots \times n$ ``base'' grid in $\R^{d}$, and choose a point in $\R^{d-1}$ for each intersection point in the grid. We then form convex sets in $\R^{2d-1}$ by taking the convex hull of those points lying ``above'' a hyperplane in the $d$-dimensional grid and we group each collection of parallel sets into a family. (So we get $d$ families overall.)

More formally, we choose vectors $x^1,\dots,x^d \in \R^n$ with $x^i_1 < x^i_2 < \cdots < x^i_n$ and a point $z_\mathbf{t} \in \R^{d-1}$ for each $\mathbf t = (t_1,\dots,t_d) \in [n]^d$. Then we set $P_{\mathbf t} = (x^1_{t_1}, \dots, x^d_{t_d}, z_\mathbf{t}) \in \R^{2d-1}$, and we form the convex sets
\[
    A^i_j = \conv( P_{\mathbf{t}} : t_i = j )
\]
and the families $\mathcal F^i := \{ A^i_j : 1 \leq j \leq n\}$. 

Why do we choose a point in $\R^{d-1}$ for each intersection point, instead of $\R^1$? If we were to use $\R^1$, then it would be trivial to prove the analogue of Theorem \ref{th:Helly} in higher dimensions: Restricting to a 2-dimensional subgrid of the base space recreates the 3-dimensional scenario of Section \ref{sec:main-thm}. To avoid a trivial reduction like this, we need $d-1$ additional dimensions.

There is an extrinsic reason for this choice of dimension, as well: According to Holmsen's extension \cite{Holm} of the Montejano--Karasev theorem~\cite{KaMo}, if $n=3$, then one of the families has a line transversal. The goal of this section is to prove a strengthened version of this statement in our setting---the analogue of Lemma \ref{thm:3x3_grid} for higher dimensions.

\begin{proposition}
    Set $n=3$ in the setup above. For some $i$, there is a line whose first $d$ coordinates are $(x^1_2,x^2_2,\dots,x^{i-1}_2,\,\cdot\,,x^{i+1}_2,\dots,x^d_2)$ that pierces $\mathcal F^i = \{A^i_1,A^i_2,A^i_3\}$.\footnote{In other words, the line lies ``above'' one of the central lines in the base $3\times 3 \times \cdots \times 3$ grid.}
\end{proposition}
\begin{proof}
    Let
    \[
        B^i = \big\{ (x^1_2,x^2_2,\dots,x^{i-1}_2, y ,x^{i+1}_2,\dots,x^d_2) : y \in \R\} \times \R^{d-1}.
    \]
    We are looking for an $i$ such that $B^i$ contains a line that pierces $\mathcal F^i$.
    Say that $x^i_2 = \alpha^i_1 x^i_1 + \alpha^i_3x^i_3$ where $\alpha^i_1 + \alpha^i_3 = 1$ and $\alpha^i_1,\alpha^i_3 \geq 0$. (The values $\alpha^i_r$ are determined uniquely.) Given $J \subseteq [d]$ and $\mathbf r \in \{1,3\}^J$,
    define $\mathbf{t}_{J,\mathbf r} \in [n]^d$ by 
    \[
        (\mathbf{t}_{J,\mathbf{r}})_i = \begin{cases}
            2 &\text{if } i \notin J\\
            \mathbf r_i &\text{if } i \in J
        \end{cases}
    \]
    so for example $\mathbf{t}_{[d],\mathbf r} = \mathbf r$ and $\mathbf{t}_{\emptyset,\mathbf r}=\mathbf 2$.
    Also let $P_{J,\mathbf{r}} := P_{\mathbf{t}_{J,\mathbf{r}}}$.
    Now, for each $J \subseteq [d]$, define the point
    \[
        Q_J = \sum_{\mathbf{r} \in \{1,3\}^J} \Big( \prod_{j \in J} \alpha^j_{\mathbf{r}_j} \Big) P_{J, \mathbf{r}}.
    \]
    Since $\sum_{\mathbf{r} \in \{1,3\}^J} \Big( \prod_{j \in J} \alpha^j_{\mathbf{r}_j} \Big) = \prod_{j \in J} (\alpha^j_1 + \alpha^j_3) = 1$, the point $Q_J$ is a convex combination of the points $P_{J,\mathbf{r}}$.
    
    The coefficients are chosen in this convex combination so that the first $d$ coordinates of $Q_J$ are $(x^1_2,x^2_2,\dots,x^d_2)$. Moreover, we have two properties of $Q_J$: If $i \notin J$, then $P_{J,\mathbf{r}} \in A^i_2$ for every $\mathbf r \in \{1,3\}^J$, so $Q_J \in A^i_2$. On the other hand, if $i \in J$, then
    \[
        Q_J =
        \alpha^i_1 \sum_{\substack{\mathbf{r} \in \{1,3\}^{J} \\ \mathbf{r}_i = 1}} \Big( \prod_{j \in J\setminus \{i\}} \! \alpha^i_{\mathbf{r}_j} \Big) P_{J, \mathbf{r}} + 
        \alpha^i_3 \sum_{\substack{\mathbf{r} \in \{1,3\}^{J} \\ \mathbf{r}_i = 3}} \Big( \prod_{j \in J\setminus \{i\}} \! \alpha^i_{\mathbf{r}_j} \Big) P_{J, \mathbf{r}},
    \]
    which is a convex combination of a point in $B^i \cap A^i_1$ and a point in $B^i \cap A^i_3$.

    If there is an $i$ such that
    \[
        \conv( Q_J : i \in J) \cap \conv(Q_J : i \notin J) \neq \emptyset,
    \]
    then we are done. To see why, let $Q$ denote the point in the intersection. On the one hand, because $Q \in \conv(Q_J : i \notin J)$, we know that $Q \in A^i_2$. On the other hand, because $Q \in \conv(Q_J : i \in J)$, $Q$ is contained in a line that pierces $A^i_1$ and $A^i_3$ but is itself contained in $B^i$.

    Since the first $d$ coordinates of $Q_J$ are the same for every $J$, the $Q_J$'s are contained in a $(d-1)$-dimensional affine subspace. Thus, the following Lemma \ref{thm:subset-points} proves that such an $i$ always exists.
\end{proof}

\begin{lemma}\label{thm:subset-points}
    If $Q_J \in \R^{d-1}$ for each $J \subseteq [d]$, there is an index $i$ such that
    \[
        \conv( Q_J : i \in J ) \cap \conv(Q_J : i \notin J) \neq \emptyset.
    \]
\end{lemma}
\begin{proof}
    Suppose the conclusion is false. Then for each $i \in [d]$, there is a hyperplane $H_x$ that separates the sets $\{ Q_J : i \in J\}$ and $\{ Q_J : i \notin J\}$. These hyperplanes divide $\R^{d-1}$ into cells, and each point $Q_J$ must lie in a separate cell. However, we know that $d$ hyperplanes divide $\R^{d-1}$ into at most
   \[
        \sum_{i=0}^{d-1} \binom{d}{i} = 2^d - 1
    \]
    cells, which is a contradiction.
\end{proof}

\begin{remark}
    We can alternatively prove Lemma \ref{thm:subset-points} via topology. The function $f\colon J \mapsto Q_J$ can be considered instead as a function $f\colon \{-1,1\}^d \to \R^{d-1}$. We can extend $f$ to a function $\tilde f$ on the boundary of the unit cube in which $\tilde f(F) = \conv\big(f(F)\big)$ for every facet $F$. By the Borsuk--Ulam theorem, there is a pair of antipodal points $u,-u \in \partial [-1,1]^d$ such that $\tilde f(u) = \tilde f(-u)$; the facets that $u$ and $-u$ belong to correspond exactly to a partition of vertices as described in Lemma \ref{thm:subset-points}.
\end{remark}

\bigskip
{\bf Acknowledgements.}
This research started during the special semester on Discrete Geometry and Convexity at the Erdős Center, Budapest, in 2023, supported by ERC Advanced Grants ``GeoScape" and ``ERMiD", no. 882971 and 101054936.
IB was partially supported by NKFIH grants No.\ 131529, 132696, and 133919 and also by the HUN-REN Research Network.
TD was supported by a National Science Foundation Graduate Research Fellowship under Grant No.\ 2141064.
DP was supported by the ERC Advanced Grant ``ERMiD'' and by the J\'anos Bolyai Research Scholarship of the Hungarian Academy of Sciences, and by the New National Excellence Program \'UNKP-23-5 and by the Thematic Excellence Program TKP2021-NKTA-62 of the National Research, Development and Innovation Office.
DV was supported by the Ministry of Innovation and Technology NRDI Office within the framework of the Artificial Intelligence National Laboratory (RRF-2.3.1-21-2022-00004).

\end{document}